\documentclass[12pt]{amsart}
\usepackage{amssymb,latexsym}
\usepackage{amsfonts,amsmath,amsthm}
\usepackage[dvips]{graphicx}
\usepackage[latin1]{inputenc}
\newtheorem*{teoa}{Theorem A}

\newtheorem*{teob}{Theorem B}

\newtheorem*{conjetura}{Conjecture}

\newtheorem{defi}{Definition} 
\newtheorem{teo}{Theorem} 
\newtheorem{pro}{Proposition} 

\newtheorem{lema}{Lemma}
\newtheorem{rem}{Remark}
\newtheorem{claim}{Claim}
\newcommand{\N}{\mathbb{N}}

\newcommand{\bs}{_{\blacksquare}}

\newcommand{\R}{\mathbb{R}}

\begin{document}

\title[A generalization of Escobar-Riemann problem]{A generalization of Escobar-Riemann mapping type problem on smooth metric measure spaces}
\author{Jhovanny Mu\~{n}oz Posso\textsuperscript{1} \textsuperscript{2}}


\maketitle


\footnotetext[1]{ {\it Universidad del Valle, Departamento de Matem\'{a}ticas, Universidad
del Valle {A. A. 25360. Cali, Colombia.}
}}

\footnotetext[2]{ {\it IMPA, Estrada Dona Castorina 110, Rio de Janeiro 22460-320, Brasil
}}

\begin{abstract}
In this article, we introduce an analogous problem to Yamabe type problem considered by Case in \cite{CaseYamabe}, which generalizes the Escobar-Riemann mapping problem for smooth metric measure spaces with boundary. The last problem will be called Escobar-Riemann mapping type problem. For this purpose, we consider the generalization of Sobolev Trace Inequality deduced by Bolley at. al. in \cite{newtrace}. This trace inequality allows us to introduce an Escobar quotient  and its infimum. This infimum  we call the  Escobar weighted constant. The Escobar-Riemann mapping type problem for  smooth metric measure spaces in manifolds with boundary consists of finding a function which attains the Escobar weighted constant. Furthemore, we resolve the later problem when Escobar weighted constant is negative.  Finally,  we get an Aubin type inequality connecting the weighted Escobar constant for compact smooth metric measure space and the optimal constant for the trace inequality in \cite{newtrace}.
\end{abstract}


\section{\bf Introduction}


When $(M^n, g)$ is a Riemannian manifold with boundary, we denote by  $\partial M$ the boundary of $M$ and by $H_g$ the trace of the second fundamental form of $\partial M$. 
The Escobar-Riemann mapping problem for manifolds with boundary is concerned with finding a metric $g$ with  scalar curvature $R_{g} \equiv 0$ in $M$ and $H_{g}$ constant  on $\partial M$, in the conformal class of the initial metric $g$. Since this problem in the Euclidean half-space reduces to finding the minimizers in  the sharp Trace Sobolev inequality, we consider a particular case of the Trace Gagliardo-Nirenberg-Sobolev inequality in \cite{newtrace}. 

To present the Trace Gagliardo-Nirenberg-Sobolev inequality, let  $\R^{n}_{+} = \{ (x, t) : x \in \R^{n-1}, t \geq 0 \}$ denote the half-space 
and  its boundary by $\partial \R^{n}_{+} = \{ (x, 0) \in \R^n : x \in \R^{n-1} \}$. We identify $\partial \R^{n}_{+}$ with   $\R^{n-1}$ whenever necessary.

\begin{teo}\cite{newtrace}\label{teotrace}
Fix $m \geq 0$. For all $w \in W^{1,2}(\R^{n}_{+}) \cap L^{\frac{2(m+n-1)}{m+n-2}} (\R^{n}_{+})$ it holds that

\begin{equation}\label{general trace Jhova}
\Lambda_{m, n} \left( \int_{\partial \R^{n}_{+}} | w| ^{\frac{2(m + n -1)}{m +n - 2}} \right) ^{\frac{2m + n -2}{m +n -1}}   \leq \left( \int_{\R^{n}_{+}} |\nabla w|^{2}  \right) \left( \int_{\R^{n}_{+}} |w|^{\frac{2(m + n - 1)}{m + n -2}} \right)^{\frac{m }{m +n -1}}
\end{equation} where the constant $\Lambda_{n, m}$ is given by

\begin{equation}\label{c trace}
\Lambda_{m,n} = (m+n-2)^2 \left( \dfrac{ Vol(S^{2m + n - 1})^{\frac{1}{2m + n - 1}}}{2(2m + n -2)} \right)^{\frac{2m + n -1}{m + n - 1}} \left( \dfrac{ \Gamma(2m + n -1)}{\pi^m \Gamma(m + n -1)} \right)^{\frac{1}{m + n - 1}}
\end{equation} and $ Vol(S^{2m + n - 1})$ is the volume of the $2m + n -1$ dimensional unit sphere. Moreover, equality holds if and only if $w$ is a constant  multiple of the function $w _{\epsilon, x_0}$ defined on $\R^{n}_{+}$ by

\begin{equation}\label{funcion trace}
w _{\epsilon, x_0} (x, t) := \left(\frac{2 \epsilon}{(\epsilon + t)^2 + |x - x_0|^2}  \right)^{\frac{m+n-2 }{2}}
\end{equation} where $\epsilon > 0 $ and $x_0 \in \R^{n-1}$.
\end{teo}

Del Pino and Dolbeaut studied the sharp Gagliardo-Nirenberg-Sobolev inequalities. Based on Del Pino and Dolbeaut's result, Case in \cite{CaseYamabe} considered a Yamabe type problem for smooth metric measure spaces in manifolds without boundary, which generalizes the Yamabe problem when $m=0$.  Then, using Theorem \ref{teotrace} instead of  Gagliardo-Nirenberg-Sobolev inequalities and following similar ideas  in \cite{CaseYamabe}, we will introduce an Escobar-Riemann mapping type problem for smooth metric measure spaces in manifolds with boundary. Thus, it is necessary to consider the notion of \textit{smooth metric measure space with boundary} defined by a five-tuple $(M^n, g, e^{-\phi}dV_g, e^{-\phi}d \sigma_g, m )$  where $dV_g$ and $d \sigma_g$ are the volume form induced by the metric $g$ in $M$ and on the boundary $\partial M$, respectively; a function $\phi$ such that $\phi \in C^{\infty}(M)$; and a parameter $m \in [0, \infty)$. In addition, if $m = 0$, we require $\phi = 0$. 

Let us denote the scalar curvature, the Laplacian and the Gradient associated to the metric $g$ by $R_g$, $\Delta_g$, and $\nabla_g$, respectively. The \textit{weighted scalar curvature} $R^{m}_{\phi}$ of a smooth metric measure space for $m=0$ is $R^{m}_{\phi} = R_g$ and for $m \neq 0$ is the function $R^{m}_{\phi} := R_g + 2\Delta_g \phi- \frac{m+1}{m}|\nabla_g \phi|^{2}$. The  \textit{weighted Escobar quotient} for this smooth metric measure is defined by

\begin{equation}\label{quotient}
\begin{array}{ll}
\mathcal{Q}(w) = & \dfrac{\int_{M} (|\nabla w|^{2} + \frac{m + n - 2}{4(m + n - 1)}R^{m}_{\phi } w^{2}) e^{-\phi} dV_g  + \int_{\partial M} \frac{m + n - 2}{2(m + n - 1)} H^m_{\phi} w^{2} e^{-\phi} d \sigma_g  }{ (\int_{\partial M } |w|^{\frac{2(m + n - 1)}{m +n - 2}}e^{-\phi} d\sigma_g)^{\frac{2m + n -2}{m +n -1}} (\int_{M}| w| ^{\frac{2(m + n - 1)}{m +n - 2}}e^{-\frac{(m-1)\phi}{m}}dV_g) ^{-\frac{m }{m +n -1}}},
\end{array}  
\end{equation} where we denote by $H^m_{\phi} = H_g + \dfrac{\partial \phi}{\partial \eta}$ the Gromov mean curvature and  $\dfrac{\partial }{\partial \eta}$ is the outer normal derivative. 

The \textit{weighted Escobar constant} $\Lambda[M^n, g, e^{-\phi} dV_g, e^{-\phi} d \sigma_g, m] \in \R \cup \{ -\infty \}$  is defined by

\begin{equation}
\Lambda:=\Lambda[M^n, g, e^{-\phi} dV_g, e^{-\phi} d \sigma_g, m ] = \inf \{  \mathcal{Q}(w) : w \in H^1(M, e^{-\phi} dV_g)\}.
\end{equation}

If $m = 0$, the quotient \eqref{quotient} coincides with the Sobolev quotient considered by Escobar in the Escobar-Riemann mapping problem. We prove the existence of a minimizer of the weighted Escobar constant when this constant is negative. The exact statement is

\begin{teoa}\label{minimizacion caso negativo}
Let $(M^n, g, e^{-\phi} dV_g, e^{-\phi} d \sigma_g, m)$ be a compact smooth metric measure space with boundary, $m \geq 0$ and negative weighted Escobar constant. Then there exists a positive function $w \in C^{\infty}(M)$ such that 

\[
\mathcal{Q}(w) = \Lambda[M^n, g, e^{-\phi} dV_g, e^{-\phi} d \sigma_g, m ].
\]
\end{teoa}

Using Theorem \ref{teotrace}, we prove that the weighted Escobar constant for a compact smooth measure space with boundary is always less or equal than the weighted Escobar constant of the model case $(\R^n_{+}, dt^2 + dx^2 , dV, d\sigma, m)$.

\begin{teob}\label{Jhova Aubin desigualdad no estricta}
Let $(M^n, g, e^{-\phi} dV_g, e^{-\phi} d\sigma_g, m)$ be a compact smooth metric measure space with boundary such that $m \geq 0$. Then

\begin{equation}\label{desigualdad Jhova Aubin}
\Lambda[M^n, g, e^{-\phi} dV_g, m] \leq \Lambda[\R^n_{+}, dt^2 + dx^2 , dV, d\sigma, m] = \Lambda_{m, n}.
\end{equation}
\end{teob}

We recall that  in the Escobar-Riemann mapping problem ($m=0$) if the inequality \eqref{desigualdad Jhova Aubin} is strict, it follows the existence of the minimizer. The same result is expected  for the Escobar-Riemann type problem. For that reason we conjecture that

\begin{conjetura}
Let $(M^n, g, e^{-\phi} dV_g, e^{-\phi} d\sigma_g, m)$ be a compact smooth metric measure space with boundary such that $m  \geq 0$ and

\begin{equation}
\Lambda[M^n, g, e^{-\phi} dV_g, m] < \Lambda[\R^n_{+}, dt^2 + dx^2 , dV, d\sigma, m] = \Lambda_{m, n}.
\end{equation}

Then there exists a positive function $w \in C^{\infty}(M)$ such that

\[
\mathcal{Q}(w) = \Lambda[M^n, g, e^{-\phi} dV_g, e^{-\phi} d \sigma_g, m ].
\]

\end{conjetura}

This paper is organized as follows. In section \ref{General Trace Inequality}, we give a different  proof for Theorem \ref{teotrace} given in \cite{newtrace},   for the particular case $m \in \N \cup \{ 0 \}$.\footnote{After we posted on arXiv the previous version of this paper, we were informed by  Nguyen Van Hoang that Theorem \ref{teotrace} is a particular case of Theorem 18 in \cite{newtrace}.} In sections \ref{Smooth metric measure spaces with boundary} and \ref{Preliminaries},   we consider our notion smooth metric measure spaces with boundary and other concepts to introduce Escobar-Riemann type problem. In sections \ref{Theorem A} and \ref{Theorem B}, we prove Theorem A and B, respectively.

\section{General Trace Inequality}\label{General Trace Inequality}


In this section, we give a  proof for Theorem \ref{teotrace} in the case $m \in \N \cup \{ 0 \}$ different to the proof  in \cite{newtrace}. As we mentioned in the introduction, the Trace Gagliardo-Nirenberg-Sobolev inequality prepares the way to introduce our Escobar-Riemann type problem. The proof that we present depends on the Sobolev Trace Inequality in $\R^{n + 2m}$ and its minimizers. This kind of ideas are due to Bakry et al. (see \cite{Bakry}).

\begin{rem} In the case $m = 0$ in the inequality \eqref{general trace Jhova} we recover the Sobolev trace inequality (see \cite{Becknertrace}, \cite{Escobartrace}) 

\begin{equation}\label{Escobartrace inequality}
\Lambda_{0, n} \left( \int_{\partial \R^{n}_{+}} | w| ^{\frac{2(n -1)}{n - 2}} \right) ^{\frac{n -2}{n -1}}   \leq \left(
\int_{\R^{n}_{+}} |\nabla w|^{2}  \right),
\end{equation} where $\Lambda_{0, n} = \frac{n-2}{2} (vol(S^{n-1}))^{\frac{1}{n-1}}$. Equality in \eqref{Escobartrace inequality}  holds if and only if $w$ is a positive constant multiple of the functions of the form 

\begin{equation}\label{Escobar burbuja}
 w = \left(\frac{\epsilon}{(\epsilon + t)^2 + |x-x_0|^2} \right)^{\frac{n-2}{2}}.
\end{equation}
\end{rem}

\begin{lema}\label{minimizacion h}
Let $p, q, B, C$ be positive numbers and define $h(\tau) = B \tau^{p} + C\tau^{-q}$ for $\tau > 0$. Then $h$ attains the infimum in $\tau_0 = (\frac{qB}{pA})^{\frac{1}{p+q}}$ and 

\[
\inf_{\tau > 0}h(\tau) = h(\tau_0) = B^{\frac{q}{p+q}} C^{\frac{p}{p+q}} \left({\frac{q}{p}} \right)^{\frac{p}{p+q}} \left({\frac{q + p}{p}}\right).
\]

\end{lema}

\begin{proof} Since $h$ is a positive continuous function for $\tau > 0$ and

$$\lim \limits_{\tau \to 0^{+}} h(\tau) = \lim \limits_{\tau \to \infty} h(\tau) = \infty, 
$$ it follows that $h$ attains the infimum for some $\tau_0 > 0$. A direct computation shows that $h'(\tau) = \tau^{p-1}(pB - qC\tau^{-p-q})$. Therefore $\tau_0 = (\frac{qC}{pB})^{\frac{1}{p+q}}$ and

\begin{equation}
\begin{array}{ll}
h((\frac{qC}{pB})^{\frac{1}{p+q}}) &= B (\frac{qC}{pB})^{\frac{p}{p+q}} + C(\frac{qC}{pB})^{\frac{-q}{p+q}} \vspace{0.2cm} \\
&= B^{\frac{q}{p+q}} C^{\frac{p}{p+q}}(\frac{q}{p})^{\frac{p}{p+q}} + B^{\frac{q}{p+q}} C^{\frac{p}{p+q}}(\frac{q}{p})^{\frac{-q}{p+q}} \vspace{0.2cm} \\
&= B^{\frac{q}{p+q}} C^{\frac{p}{p+q}}(\frac{q}{p})^{\frac{p}{p+q}} (1 + \frac{p}{q}) \vspace{0.2cm} \\
&= B^{\frac{q}{p+q}} C^{\frac{p}{p+q}}(\frac{q}{p})^{\frac{p}{p+q}} (\frac{p+ q}{q}). 
\end{array}
\end{equation}

\end{proof}

\begin{rem} If $m \to \infty$, the inequality \eqref{general trace Jhova} takes the form

\begin{equation}\label{1 trace inequality}
\Lambda_{\infty, n} \left( \int_{\partial \R^{n}_{+}} | w| ^{2} \right) ^{2}  \leq \left( \int_{\R^{n}_{+}} |\nabla w|^{2}  \right) \left( \int_{\R^{n}_{+}} |w|^{2} \right)
\end{equation}

where $\lim \limits_{m \to \infty} \Lambda_{m, n} = \Lambda_{\infty, n}$.

The inequality \eqref{1 trace inequality} is equivalent to the  trace inequality \hbox{$H^1(M) \to L^2(\partial M)$}

\begin{equation}\label{2 trace inequality}
2 (\Lambda_{\infty, n})^{\frac{1}{2}} \left( \int_{\partial \R^{n}_{+}} | w| ^{2} dx \right)   \leq \int_{\R^{n}_{+}} |\nabla w|^{2}dxdt  + \int_{\R^{n}_{+}} |w|^{2} dxdt.
\end{equation}

In fact, suppose inequality \eqref{2 trace inequality} holds. For $\tau > 0$ define the function $w_{\tau}(x,t) = w(\frac{1}{\tau}(x, t))$. The change of variable $(y,s) = \frac{1}{\tau}(x, t)$ implies

\[
\int_{\partial \R^{n}_{+}} | w_{\tau}|^{2} (x, 0) dx = \tau^{n-1} \int_{\partial \R^{n}_{+}} | w|^{2}(y, 0)  dy,
\]

\[
\int_{\R^{n}_{+}} |\nabla w_{\tau}|^{2}(x,t) dxdt  = \tau^{n-2} \int_{\R^{n}_{+}} |\nabla w|^{2}(y,s)  dyds
\]

and 

\[
\int_{\R^{n}_{+}} | w_{\tau}|^{2}(x,t) dxdt  = \tau^{n} \int_{\R^{n}_{+}} | w|^{2} (y,s)  dyds.
\]

Then, using $w_{\tau}$ and the  equalities above in inequality \eqref{2 trace inequality} we get

\begin{equation}\label{3 trace inequality}
2(\Lambda_{\infty, n})^{\frac{1}{2}} \left( \int_{\partial \R^{n}_{+}} | w| ^{2}(y, 0) dy \right)   \leq \tau B  + \tau^{-1}C,
\end{equation}

where $B = \int_{\R^{n}_{+}} | w|^{2}(y,s)  dyds$ and $C = \int_{\R^{n}_{+}} |\nabla w|^{2}(y,s)  dyds$. Lemma \ref{minimizacion h} yields that for $\tau_0 = (\frac{C}{B})^{\frac{1}{2}}$, it holds

\begin{equation}\label{4 trace inequality}
 \tau_0 B  + \tau_0^{-1}C = 2B^{\frac{1}{2}}C^{\frac{1}{2}} = 2 \left( \int_{\R^{n}_{+}} |\nabla w|^{2} dxdt \right)^{\frac{1}{2}} \left( \int_{\R^{n}_{+}} |w|^{2} dxdt\right)^{\frac{1}{2}}.
\end{equation}
 
Since inequality \eqref{3 trace inequality} is true for every $\tau > 0$, in particular it is true for $\tau_0 = (\frac{C}{B})^{\frac{1}{2}}$ and by \eqref{4 trace inequality}, we have

\begin{equation}\label{5 trace inequality}
2 (\Lambda_{\infty, n})^{\frac{1}{2}} \left( \int_{\partial \R^{n}_{+}} | w| ^{2} \right)   \leq 2 \left( \int_{\R^{n}_{+}} |\nabla w|^{2}  \right)^{\frac{1}{2}} \left( \int_{\R^{n}_{+}} |w|^{2} \right)^{\frac{1}{2}},
\end{equation}

which is equivalent to \eqref{1 trace inequality}.

Now, suppose that inequality \eqref{1 trace inequality}  holds, then inequality \eqref{5 trace inequality} holds. In addition, inequality \eqref{2 trace inequality} is a consequence of inequality $2ab \leq a^2 + b^2$.

\end{rem}

In our proof for the Theorem \ref{teotrace}, we use the following Lemma, which was taken from \cite{CaseYamabe}.

\begin{lema}\label{integral Case}
Fix $k$, $l \geq 0$, $2m \in \mathbb{N}$, and constants $a$, $\tau > 0$. Then

\[
\int_{\R^{2m}} \dfrac{|y|^{2l} dy}{(a + \frac{|y|^{2}}{\tau}) ^{2m + k}} = \dfrac{\pi^{m}\Gamma(m+l)\Gamma(m+ k-l) \tau^{m+l}}{ \Gamma(m) \Gamma(2m+k) a^{m+k-l}}.
\]
\end{lema}

\bigskip

\textbf{\textit{Proof of Theorem \ref{teotrace}}}. We are able to prove inequality \eqref{general trace Jhova} only for $m \in \mathbb{N}$. For this purpose, consider the inequality \eqref{Escobartrace inequality} for $\R^{n+2m}_+$. The   idea of the proof consists of using this inequality for the special function

\begin{equation}\label{f trace}
f(y, x, t) := \left(w^{\frac{- 2}{m + n - 2}} (x,t) + \frac{|y|^{2}}{\tau} \right)^{-\frac{2m+n-2}{2}} \in C^{\infty}(\R^{n + 2m}_{+}),
\end{equation}

where $ (x, t) \in \R^{n}_{+}$, $y \in \R^{2m}$ and $\tau > 0$.

Suppose $f$ is of the form \eqref{f trace}. First, we analyze the term on the left hand side of inequality  \eqref{Escobartrace inequality}.  Fixing $(x,t)$ we note that $f^{\frac{2(2m + n -1)}{2m + n -2}}$ takes the form of the function considered in  Lemma \ref{integral Case} with $a = w^{\frac{- 2}{m + n - 2}} (x,t)$. Fubini's Theorem, Lemma \ref{integral Case} with $k = n-1$ and $l = 0$, and some calculation yield 

\begin{equation}\label{bordo trace}
\int_{\partial \R^{2m + n}_{+}} f^{\frac{2(2m + n -1)}{2m + n -2}} dxdy= \dfrac{\pi^{m}\Gamma(m+n-1) \tau^{m}}{ \Gamma(2m + n - 1) } \int_{\partial \R^{n}_{+}} w^{\frac{2(m + n - 1)}{m + n - 2}}dx.
\end{equation}

In order to analyze  the term on the right  hand side of inequality \eqref{Escobartrace inequality},  we compute

\[
|\nabla f|^2 = \dfrac{ \left( \frac{2m + n - 2}{2} \right)^2 \left( \left(\dfrac{2}{m + n - 2} \right)^2  w^{- \frac{2(m + n)}{m + n -2}} |\nabla w|^{2} + 4 \frac{|y|^{2}}{\tau^2} \right)}{\left( w^{-\frac{2}{m + n - 2}} + \frac{|y|^{2}}{\tau} \right)^{2m + n}}.
\]

Lemma \ref{integral Case} leads to

\small

\begin{equation}\label{gradiente trace}
\begin{array}{ll}
\displaystyle \int_{\R^{2m + n}_{+}} |\nabla f|^2 dydxdt & = \displaystyle \left( \frac{2m + n - 2}{m + n -2} \right)^2 \left( \frac{\pi^m \tau^m \Gamma(m+n)}{\Gamma(2m + n)} \right) \int_{\R^{n}_{+}}| \nabla w|^{2} dxdt \vspace{0.2cm}\\
&  \displaystyle +  \left( \frac{m (2m + n - 2)^{2} \pi^{m} \tau^{m-1} \Gamma(m+n)}{ (m + n - 1) \Gamma(2m + n)} \right) \int_{\R^{n}_{+}}|w|^{\frac{2(m+n-1)}{m+n-2}}dxdt.
\end{array}
\end{equation}

\normalsize

Using equalities \eqref{bordo trace} and \eqref{gradiente trace} in inequality \eqref{Escobartrace inequality}, we get that

\begin{equation}\label{f evaluada Sobolev trace}
\begin{array}{r}
 \displaystyle \Lambda_{2m + n, 0}  \left( \dfrac{\pi^{m}\Gamma(m+n-1) \tau^{m}}{ \Gamma(2m + n - 1) } \int_{\partial \R^{n}_{+}} w^{\frac{2(m + n - 1)}{m + n - 2}} dx \right)^{\frac{2m + n - 2}{2m + n - 1}} \vspace{0.2cm}  \\
\displaystyle \leq  \left( \frac{2m + n - 2}{m + n -2} \right)^2 \left( \frac{\pi^m \tau^m \Gamma(m+n)}{\Gamma(2m + n)} \right) \int_{\R^{n}_{+}}| \nabla w|^{2} dxdt \vspace{0.2cm}\\
\displaystyle +  \left( \frac{m (2m + n - 2)^{2} \pi^m \tau^{m-1} \Gamma(m+n)}{ (m + n - 1) \Gamma(2m + n)} \right) \int_{\R^{n}_{+}}|w|^{\frac{2(m+n-1)}{m+n-2}} dxdt.
\end{array}
\end{equation}

Rewriting \eqref{f evaluada Sobolev trace}, we obtain

\begin{equation}\label{f evaluada Sobolev trace 2}
\begin{array}{r}
 \displaystyle \Lambda_{2m + n, 0}  \left( \dfrac{\pi^{m}\Gamma(m+n-1)}{ \Gamma(2m + n - 1) } \int_{\partial \R^{n}_{+}} w^{\frac{2(m + n - 1)}{m + n - 2}} dx \right)^{\frac{2m + n - 2}{2m + n - 1}} A  \leq  h(\tau),
\end{array}
\end{equation}

where

\[
A =  \frac{ \Gamma(2m + n)}{(2m + n - 2)^{2} \pi^m \Gamma(m+n)}, 
\]

\[
h(\tau) = B \tau^{\frac{m}{2m+n-1}} + C\tau^{-\frac{m+n-1}{2m+n-1}},
\]

\[
B = \displaystyle \frac{1}{(m + n -2)^2} \int_{\R^{n}_{+}}| \nabla w|^{2} dxdt, 
\]

and

\[
C = \displaystyle \frac{m }{m + n - 1} \int_{\R^{n}_{+}}|w|^{\frac{2(m+n-1)}{m+n-2}} dxdt.
\]

Lemma \ref{minimizacion h} implies that the function $h$ minimizes for $ \tau_0 = (\frac{(m+n-1)C}{mB})^{\frac{m+n-2}{2m+n-1}} $ and

\begin{equation}\label{f evaluada Sobolev trace 3}
\begin{array}{r}
 \displaystyle \Lambda_{2m + n, 0}  \left( \dfrac{\pi^{m}\Gamma(m+n-1)}{ \Gamma(2m + n - 1) } \int_{\partial \R^{n}_{+}} w^{\frac{2(m + n - 1)}{m + n - 2}} dx \right)^{\frac{2m + n - 2}{2m + n - 1}} A  \leq  h(\tau_0).
\end{array}
\end{equation}

Inequality \eqref{f evaluada Sobolev trace 3} proves inequality \eqref{general trace Jhova}  with $\Lambda_{m,n}$ as in \eqref{c trace}. Next, we characterize the functions that achieve equality in \eqref{general trace Jhova}. Note that for $\R^{n+2m}_+$ and $f$ defined in \eqref{f trace},  the  equality in \eqref{Escobartrace inequality}    holds if and only if

\[
f(y, x, t) =  \left( \frac{(t + \epsilon)^2 + |x-x_0|^2 + |y|^{2}}{\tau} \right)^{-\frac{2m+n-2}{2}},  \quad \text{for} \quad \tau > 0,
\]

i.e

$$w^{\frac{- 2}{m + n - 2}} (x,t) = \tau ((t + \epsilon)^{2} + |x-x_0|^{2}) $$ (see Escobar \cite{Escobartrace} and Beckner \cite{Becknertrace}). Then, the family of functions $\{w _{\epsilon, x_0} \}$  in \eqref{funcion trace} is the only one that satisfies the equality in \eqref{general trace Jhova}. $\bs$



\section{Smooth metric measure spaces with boundary and the conformal Laplacian}\label{Smooth metric measure spaces with boundary}

Our approach is based on  \cite{CaseYamabe} and \cite{CaseGNS}. The first step is to introduce the definition of a smooth metric measure space with boundary

\begin{defi}  Let  $(M^n, g)$ be a Riemannian manifold and let us denote by $dV_g$ and $d \sigma_g$ 
the volume form induced by $g$ in $M$ and $\partial M$, respectively. Set a function $\phi$ such that $\phi \in C^{\infty}(M)$ and $m \in [0, \infty)$ be a  dimensional parameter. In the case $m = 0$, we require that $\phi = 0$. A smooth metric measure space with boundary is the five-tuple  \hbox{$(M^n, g, e^{- \phi} dV_g, e^{- \phi} d \sigma_g, m )$}.
\end{defi}

As in \cite{CaseYamabe}, sometimes we  denote by the four-tuple $(M^n, g, v^{m} dV_g, v^{m} d\sigma_g)$ a smooth metric measure space where $v$ and $\phi$ are related by $v^{m} = e^{- \phi}$. We denote by $R_g$ the scalar curvature of $(M, g)$ and  $Ric$ and the Ricci tensor of $(M, g)$,  $\eta$ the outer normal on $\partial M$ and $\dfrac{\partial }{\partial \eta}$ the normal derivative. Also, we denote the second fundamental form, the trace of the second fundamental form, and the mean curvature on the boundary $\partial M$,  by $h_{ij}$, $H_g: = g^{ij}h_{ij}$, and $h_g = \frac{H_g}{n-1}$;  respectively. 

\begin{defi}
Given a smooth metric measure space \hbox{$(M^n, g, e^{- \phi} dV_g, e^{- \phi} d \sigma_g, m )$}. The   weighted scalar curvature $R^{m}_{\phi}$ and the Bakry-\'Emery Ricci curvature $Ric^{m}_{\phi}$ are the tensors

\begin{equation}
R^{m}_{\phi} := R_g + 2 \Delta \phi - \frac{m + 1}{m} |\nabla \phi|^2
\end{equation} and

\begin{equation}
Ric^{m}_{\phi} := Ric +  \nabla^{2} \phi - \frac{1}{m} d \phi \otimes d\phi.
\end{equation} 

\end{defi}

\begin{defi}\label{conformal definition}
Let $(M^n, g, e^{- \phi} dV_g, e^{- \phi} d \sigma_g, m )$ and
$(M^n, \hat{g}, e^{- \hat{\phi}} dV_{\hat{g}}, e^{- \hat{\phi}} d \sigma_{\hat{g}}, m )$ be smooth metric measure spaces with boundary. We say they are pointwise conformally equivalent if there is a function $\sigma \in C^{\infty}(M)$ such that

\small
\begin{equation}\label{def conformal}
(M^n, \hat{g}, e^{- \hat{\phi}} dV_{\hat{g}}, e^{- \hat{\phi}} d \sigma_{\hat{g}}, m )
= (M^n, e^{\frac{2 \sigma}{m + n - 2}}g, e^{\frac{m + n }{m + n - 2} \sigma} e^{- \phi} dV_g, e^{\frac{m + n - 1}{m + n - 2} \sigma}e^{- \phi} d \sigma_g, m ).
\end{equation}
\normalsize

$(M^n, g, e^{- \phi} dV_g, e^{- \phi} d \sigma_g, m )$ and
$(\hat{M}^n, \hat{g}, e^{- \hat{\phi}} dV_{\hat{g}}, e^{- \hat{\phi}} d \sigma_{\hat{g}}, m )$ are conformally equivalent if there
is a diffeomorphism $F : \hat{M} \to M$ such that the new smooth metric measure space with boundary
$(F^{-1}(M), F^{*}g, F^{*}(e^{- \phi} dV_g), F^{*}(e^{- \phi} d \sigma_g), m )$ is pointwise conformally equivalent to $(\hat{M}^n, \hat{g}, e^{- \hat{\phi}} dV_{\hat{g}}, e^{- \hat{\phi}} d \sigma_{\hat{g}}, m )$.
\end{defi}

\begin{defi} Given a smooth metric measure space $(M^n, g, e^{- \phi} dV_g,  m )$. The weighted Laplacian $\Delta_{\phi}: C^{\infty}(M) \to C^{\infty}(M)$ is an operator defined by

\[
\Delta_{\phi} u = \Delta u  - \nabla u  \cdot \nabla \phi
\] where $u \in C^{\infty}(M)$,  $\Delta$ is the usual Laplacian associated to the metric $g$ and $\nabla $ is gradient calculated in the metric $g$.
\end{defi}

\begin{defi}
 Given a smooth metric measure space $(M^n, g, e^{- \phi} dV_g, e^{- \phi} d \sigma_g, m )$. The
weighted conformal Laplacian $(L^m_{\phi}, B^m_{\phi})$ is given by the interior operator and boundary operator

\begin{equation}
\begin{array}{l}
L^m_{\phi} = - \Delta_{\phi	} + \dfrac{m+ n - 2}{4(m + n - 1)} R^{m}_{\phi} \quad \quad \text{in}  \quad M, \quad \quad \\
B^m_{\phi}  = \dfrac{\partial }{\partial \eta}  + \dfrac{m+ n - 2}{2(m + n - 1)}H^m_{\phi} \quad\quad\text{on} \quad \partial M.
\end{array}
\end{equation}

\end{defi}

\begin{pro}\label{conformal laplacian}
Let \small $(M^n, g, e^{- \phi} dV_g, m)$ \normalsize and \small $(M^n, \hat{g}, e^{- \hat{\phi}} dV_{\hat{g}},  m ) $ \normalsize be two pointwise conformally equivalent smooth metric measure space   such that $\hat{g} = e^{\frac{2 \sigma}{m + n - 2}}g$ 
and $\hat{\phi} = \frac{-m\sigma}{m+n-2} + \phi$. Let us denote by $L^m_{\phi}$ and $\hat{L}^m_{\hat{\phi}}$ their respective weighted conformal Laplacians. Similarly, we denote with hat all quantities computed with respect to the  smooth metric measure space $(M^n, \hat{g}, e^{- \hat{\phi}} dV_{\hat{g}}, e^{- \hat{\phi}} d \sigma_{\hat{g}}, m ) $. Then we have $\hat{v} =  e^{\frac{\sigma}{m+n-2}}v$ and the following transformation rules

\begin{equation}\label{cambio en laplaciano conforme}
\hat{L}^m_{\hat{\phi}} (w) = e^{-\frac{m + n  + 2}{2(m + n - 2)} \sigma} L^m_{\phi} (e^{ \frac{\sigma}{2}} w), \quad \quad  \quad \quad \hat{B}^{m}_{\hat{\phi}}(w) = e^{-\frac{m + n}{2(m + n - 2)} \sigma}B^{m}_{\phi}(e^{ \frac{\sigma}{2}} w).
\end{equation}

\end{pro}

We mention that the identity in the left hand size of \eqref{cambio en laplaciano conforme} appears in \cite{CaseYamabe}. On the other hand, we denote by $(w, \varphi)_{M} = \int_{M} w. \varphi \, v^{m} dV_g$ the inner product in $L^{2}(M, v^{m} dV_g)$. Also, we denote by $||.||_{2,M}$ the norm in the  space $L^{2}(M, v^{m} dV_g)$, in some case we use the notation  $||.||$ for this norm. $H^1(M, v^{m} dV_g)$ denotes the closure of $C^{\infty} (M)$ with respect to the norm

\[
\int_{M}|\nabla w|^{2} + |w|^{2}.
\]

Here and subsequently the integrals  are computed using the measure $v^{m} dV_g$.

\section{Preliminaries for Escobar-Riemann type problem}\label{Preliminaries}

In this section, we define the weighted Escobar quotient which generalizes the quotient considered by Escobar in \cite{EscobarAn} and we consider a suitable $\mathcal{W}$-functional. In general, the weighted Escobar quotient is not necessarily finite. Similarly to \cite{CaseYamabe}, we define the energies of these functionals and we give some of their properties.

\subsection{ \bf The weighted Escobar quotient}

We start with the definition of the weighted Escobar quotient

\begin{defi}
Let $(M^n, g, v^{m} dV_g, v^{m} d \sigma_g)$ be a compact smooth metric measure space with boundary. The \textit{weighted Escobar quotient} $\mathcal{Q}: C^{\infty} (M) \to \R$ is defined by

\begin{equation}\label{Yamabe funcional}
\begin{array}{ll}
\mathcal{Q}(w) = & \dfrac{((L^{m}_{\phi} w, w)_M + (B^{m}_{\phi} w, w)_{\partial M}) (\int_{M}| w| ^{\frac{2(m + n - 1)}{m +n - 2}}v^{-1}) ^{\frac{m }{m +n -1}}}{ (\int_{\partial M } |w|^{\frac{2(m + n -1)}{m +n -2}} )^{\frac{2m + n -2}{m +n -1}}}.
\end{array}  
\end{equation}


The \textit{weighted Escobar constant} $\Lambda[M^n, g, v^{m} dV_g, v^{m} d \sigma_g] \in \R$ of the smooth metric measure space $(M^n, g, v^{m} dV_g, v^{m} d \sigma_g, m )$  is

\begin{equation}
\Lambda[M^n, g, v^{m} dV_g, v^{m} d \sigma_g, m ] = \inf \{  \mathcal{Q}(w) : w \in H^1(M, v^{m} dV_g, v^{m} d \sigma_g)\}.
\end{equation}

\end{defi}

\begin{rem}\label{simple notacion}
In some cases, when the context is clear, we will not write the dependence of the smooth metric measure space with boundary, for example we write 
$\mathcal{Q}$ and $\Lambda$ instead of $\mathcal{Q}[M^n, g, v^{m} dV_g, v^{m} d \sigma_g]$ and $\Lambda[M^n, g, v^{m} dV_g, v^{m} d \sigma_g]$, respectively. We note that since $C^{\infty} (M)$ is dense in $H^1(M, v^{m} dV_g)$ and $\mathcal{Q}(|w|) = \mathcal{Q}(w)$, it is sufficient to consider the weighted Escobar constant by minimizing over the space of non-negative smooth functions on $M$, subsequently we will do this assumption without further comment. 
\end{rem}

Now, note that the weighted Escobar quotient is  conformal in the sense of  Definition \ref{conformal definition}.

On the other hand, the weighted Escobar quotient satisfies similar properties to the weighted Yamabe quotient introduced by Case in \cite{CaseYamabe}, for example we observe that the weighted Escobar quotient is continuous in $m \in [0, \infty)$ and it is conformal in the sense of the Definition \ref{conformal definition}.

\begin{pro} Let $(M^{n}, g)$ be a compact Riemannian manifold with boundary. Fix $\phi \in C^{\infty} (M)$ and $m \in [0,\infty)$. Given any $w \in C^{\infty} (M)$, it holds that

\begin{equation}
\begin{array}{l}
\lim \limits_{k \to m} \mathcal{Q}[M^n, g, e^{- \phi} dV_g, e^{- \phi} d \sigma_g, k] (w) 
= \mathcal{Q}[M^n, g, e^{- \phi} dV_g, e^{- \phi} d \sigma_g, m](w).
\end{array}
\end{equation}

\end{pro}

\begin{pro}\label{weighted Yamabe quotient conformal}
Let $(M^n, g, v^{m} dV_g, v^{m} d \sigma_g)$ be a compact smooth metric measure space with boundary.
For any $ \sigma$, $w \in C^{\infty}(M) $ it holds that

\begin{equation}
\begin{array}{l}
\mathcal{Q}[M^n, e^{\frac{2}{m + n - 2} \sigma}g, e^{\frac{m + n}{m + n - 2} \sigma}v^{m} dV_g, e^{\frac{m + n - 1}{m + n - 2} \sigma} v^{m} d \sigma_g](w) \\
= \mathcal{Q}[M^n, g, v^{m} dV_g, v^{m} d \sigma_g](e^{\frac{\sigma}{2}} w).
\end{array}
\end{equation}

\end{pro}

\begin{proof} A straightforward computation shows that the integrals 

\begin{equation}
\int_{M} |w|^{\frac{2(m+n-1)}{m+n-2}} v^{m-1} dV_g  \quad \quad \text{and} \quad \quad \int_{\partial M} |w|^{\frac{2(m+n-1)}{m+n-2}} v^{m} d \sigma_g
\end{equation} are invariant under the conformal transformation

\begin{equation}\label{transformacion conforme}
(g, v^{m} dV_g, v^{m} d \sigma_g, w)  \to (e^{\frac{2}{m + n - 2} \sigma}g, e^{\frac{m + n}{m + n - 2} \sigma}v^{m} dV_g, e^{\frac{m + n - 1}{m + n - 2} \sigma} v^{m} d \sigma_g, e^{-\frac{\sigma}{2}} w ).
\end{equation}

By Proposition  \ref{conformal laplacian} the term $(L^{m}_{\phi} w, w) + (B^{m}_{\phi} w, w)$  is invariant under \eqref{transformacion conforme}. 

\end{proof}

Similar to the smooth metric measure spaces we have some behavior for the boundary volume. Note that in the boundary the integral $\int \limits_{\partial M} |w|^{\frac{2(m+n-1)}{m+n-2}} v^{m} d \sigma_g$ measure the boundary volume $\int_{\partial M} \hat{v}^{m} d \sigma_{\hat{g}}$ of

\begin{equation}
(M^n, \hat{g}, \hat{v}^{m} dV_{\hat{g}}, \hat{v}^{m} d \sigma_{\hat{g}}, m )
= (M^n, w^{\frac{4}{m + n - 2}}g, w^{\frac{2(m + n) }{m + n - 2}} v^{m} dV_g, w^{\frac{2(m + n - 1)}{m + n - 2} } v^{m} d \sigma_g, m ).
\end{equation}

Also with the same purpose, to simplify calculus and to avoid the trivial non-compactness of the weighted Escobar-Riemann type problem, we give the next definition of the volume-normalized on the boundary.

\begin{defi}\label{boundary normalized}
Let $(M^n,  g, v^{m} dV_g, v^{m} d \sigma_g)$ be a compact smooth metric measure space with boundary. We say
that a positive function $w \in C^{\infty}(M)$ is volume-normalized on the boundary if

\[
\int \limits_{\partial M} |w|^{\frac{2(m+n-1)}{m+n-2}} v^{m} d \sigma_g = 1.
\]

\end{defi}

\subsection{ \bf$\mathcal{W}$-functional}

We introduce a $\mathcal{W}$-functional with similar properties as the $\mathcal{W}$-functional considered by Case in \cite{CaseYamabe} and Perelman in \cite{Perelman}.

\begin{defi}
Let $(M^n,  g, v^{m} dV_g, v^{m} d \sigma_g)$ be a compact smooth metric measure space with boundary. The $\mathcal{W}$-functional, $\mathcal{W}: C^{\infty}(M) \times \R^{+} \to \R$, is defined by

\small

\begin{equation}
\begin{array}{l}
\mathcal{W}(w, \tau) =  \mathcal{W}[M^n,  g, v^{m} dV_g, v^{m} d \sigma_g](w, \tau) \vspace{0.2cm}\\
= \displaystyle \tau^{\frac{m}{2(m+n-1)}} \left( (L^{m}_{\phi} w, w) + (B^{m}_{\phi} w, w) \right) +  \int_{M} \tau^{-\frac{1}{2}} w^{\frac{2(m+n-1)}{m+n-2}}v^{-1} -  \int_{\partial M}  w^{\frac{2(m + n -1)}{m+n-2}}
\end{array}
\end{equation} \normalsize when $m \in [0, \infty) $.
\end{defi}

As the weighted Escobar  quotient and the $\mathcal{W}$-functional considered by Case in \cite{CaseYamabe}, the $\mathcal{W}$-functional defined before is continuous in $m$ and conformally invariant. Additionally, we have one scale invariant in the variable $\tau$.

\begin{pro}
Let $(M^n,  g, v^{m} dV_g, v^{m} d \sigma_g)$ be a compact smooth metric measure space with boundary. Then

\[
\lim_{k \to m} \mathcal{W}[M^n, g, e^{- \phi} dV_g, e^{- \phi} d \sigma_g, k](w, \tau) = \mathcal{W}[M^n, g, e^{- \phi} dV_g, e^{- \phi} d \sigma_g, m](w, \tau).
\]

\end{pro}

\begin{pro}\label{weighted Yamabe energy conformal}
Let $(M^n,  g, v^{m} dV_g, v^{m} d \sigma_g)$ be a compact smooth metric measure space with boundary. The $\mathcal{W}$-functional is conformally invariant in its first component:

\begin{equation}\label{W conformal}
\begin{array}{l}
\mathcal{W}[M^n, e^{2 \sigma} g,  e^{(m+n)\sigma}v^{m} dV_g, e^{(m+n-1)\sigma}v^{m} d \sigma_g](w, \tau) \vspace{0.2cm} \\ 
= \mathcal{W}[M^n, g, v^{m} dV_g, v^{m} d \sigma_g](e^{\frac{(m+n -2)}{2}\sigma} w, \tau)
\end{array}
\end{equation} for all $\sigma$, $w \in C^{\infty}(M)$ and $\tau > 0$. It is scale invariant in its second component:

\begin{equation}\label{W c}
\begin{array}{l}
\mathcal{W}[M^n, cg, v^{m} dV_{cg}, v^{m} d \sigma_{cg}](w, \tau) \\
= \mathcal{W}[M^n, g, v^{m} dV_g, v^{m} d \sigma_g](c^{\frac{(n-1)(m+n-2)}{4(m+n -1)}} w, c^{-1} \tau).
\end{array}
\end{equation}

\end{pro}

\begin{proof} The equality \eqref{W conformal} follows as in Proposition \ref{weighted Yamabe quotient conformal}  and the equality \eqref{W c} follows by a direct computation. \end{proof}

Since we are interested in minimizing the weighted Escobar quotient it is natural to define the following  energies as infima using the $\mathcal{W}$-functional and relating one of these energies with the weighted Escobar constant.

\begin{defi}
Let $(M^n,  g, v^{m} dV_g, v^{m} d \sigma_g)$ be a compact smooth metric measure space with boundary. Given $\tau > 0$, the $\tau$-energy $\nu[M^n, g, v^{m} dV_g, v^{m} d \sigma_g](\tau) $ is the number defined by

\begin{equation}
\begin{array}{l}
\nu(\tau) = \nu[M^n, g, v^{m} dV_g, v^{m} d \sigma_g](\tau) \\
= \inf \left\{ \mathcal{W} (w, \tau) : w \in H^1(M, v^{m} dV_g, v^{m} d \sigma_g),  \int_{\partial M} w ^{\frac{2(m+n-1)}{m+n-2}} = 1 \right\}.
\end{array}
\end{equation}

The energy $\nu[M^n, g, v^{m} dV_g, v^{m} d \sigma_g] \in  \R \cup \{ - \infty \}$ is defined by

\[
\nu = \nu[M^n, g, v^{m} dV_g, v^{m} d \sigma_g] = \inf_{\tau >0} \nu[g, v^{m} dV_g, v^{m} d \sigma_g](\tau).
\]
\end{defi}

The conformal invariance in the  $\mathcal{W}$-functional is transferred to the energies.

\begin{pro}
Let $(M^n,  g, v^{m} dV_g, v^{m} d \sigma_g)$ be a compact smooth metric measure space with boundary. Then

\[
\nu[M^n, ce^{2 \sigma} g,  e^{(m+n)\sigma}v^{m} dV_{cg}, e^{(m+n-1)\sigma}v^{m} d \sigma_{cg}] (c\tau) = \nu[M^n, g, v^{m} dV_g, v^{m} d \sigma_g](\tau),
\]

\[
\nu[M^n, ce^{2 \sigma} g,  e^{(m+n)\sigma}v^{m} dV_{cg}, e^{(m+n-1)\sigma}v^{m} d \sigma_{cg}] = \nu[M^n, g, v^{m} dV_g, v^{m} d \sigma_g]
\] for all $\sigma \in C^{\infty}(M)$ and  $c > 0$.
\end{pro}

The following proposition shows that it is equivalent to considering the energy instead of the weighted Escobar constant when the latter is positive. 

\begin{pro}\label{W y Yamabe}
Let $(M^n,  g, v^{m} dV_g, v^{m} d \sigma_g)$ be a compact smooth metric measure space with boundary and denote by $\Lambda$ and $\nu$ the weighted Escobar constant and the energy, respectively.

\begin{itemize}
\item $\Lambda \in [-\infty, 0)$ if and only if $\nu = -\infty$;
	
 \item $\Lambda = 0$ if and only if $\nu = -1$; and
 \item $\Lambda > 0$ if and only if $\nu > -1$. Moreover, in this case we have

\begin{equation}\label{nu y lambda}
	\nu = \frac{2m + n - 1}{m} \left[\frac{m \Lambda}{m+n-1}  \right]^{\frac{m+n-1}{2m+n-1}} - 1
\end{equation}

and $w$ is a volume-normalized minimizer of $\Lambda$ if and only if $(w, \tau)$ is a volume-normalized minimizer of $\nu$ for

\begin{equation}\label{nu lambda}
	\tau = \left[ \dfrac{m \int_{M}  w^{\frac{2(m+n-1)}{m+n-2}}v^{-1} }{(m + n - 1) ((L^{m}_{\phi} w, w) + (B^{m}_{\phi} w, w))} \right]^{\frac{m+n-1}{2(2m + n -1)}}.
\end{equation}

\end{itemize}

\end{pro}

\begin{proof} If $\Lambda \in [-\infty, 0)$ then there is a volume-normalized function $w \in C^{\infty}(M)$ such that $(L^{m}_{\phi} w, w) + (B^{m}_{\phi} w, w) < 0$. Then, it is clear that $\mathcal{W}(w, \tau) \to -\infty$  as $\tau  \to \infty$ and it follows that $\nu = - \infty$. Reciprocally, if $\nu = - \infty$ there exist  a volume-normalized function  $w$ and $\tau >0$ such that $\mathcal{W}(w, \tau) < -1$, it follows that $(L^{m}_{\phi} w, w) + (B^{m}_{\phi} w, w) < 0$ and $\Lambda \in [-\infty, 0)$.

Suppose $\Lambda \geq 0$. Lemma \ref{minimizacion h} shows that if $A$, $B > 0$, then

\begin{equation}\label{Yamabe y nu}
\inf_{x > 0} \{ A x^{\frac{m}{m+n-1}} + Bx^{-1} \} = \frac{2m + n - 1}{m} \left[\frac{m}{m + n - 1} A B^{\frac{m}{m+n-1}} \right]^{\frac{m+n-1}{2m+n-1}}
\end{equation} for all $x > 0$, with equality if and only if

\begin{equation}\label{igualdad Yamabe y nu}
x= \left[ \dfrac{m B}{(m + n - 1) A} \right]^{\frac{m+n-1}{2m + n -1}}.
\end{equation} Notes that equality  \eqref{Yamabe y nu} is achieved in the case $A = 0$. Then,  from equality \eqref{Yamabe y nu}, the definitions of $\Lambda$ and $\nu$ and taking minimizing sequences of these infima we get the remain equivalences. When $\Lambda > 0$ using \eqref{Yamabe y nu} and \eqref{igualdad Yamabe y nu} we get that \eqref{nu y lambda} and \eqref{nu lambda} holds. \end{proof}

\subsection{ \bf Variational formulae for the weighted energy functionals}

The next proposition contains the computation of the Euler-Lagrange equations of the minimizing of weighted Escobar quotient. We will use it in the proof of Theorem A on the regularity part.   

\begin{pro}\label{ecuacion minimizador Yamabe}
Let $(M^n,  g, v^{m} dV_g, v^{m} d \sigma_g)$ be a compact smooth metric measure space with boundary and suppose that $0 \leq w \in H^{1}(M)$ is a volume-normalized minimizer of the weighted Escobar constant $\Lambda$. Then $w$ is a weak solution of

\begin{equation}
\begin{array}{ll}\label{Euler Lagrange Escobar quotient}
L^{m}_{\phi}w + c_{1}w^{\frac{m+n}{m+n-2}} v^{-1} = 0, &\quad \text{in} \quad M, \\
B^{m}_{\phi}w = c_{2}w^{\frac{m+n}{m+n-2}}, &\quad \text{on} \quad \partial M \\
\end{array}
\end{equation} where

\[
c_1 = \frac{m \Lambda}{m+n-2} \left(\int_{M}w^{\frac{2(m+n-1)}{m+n-2}}v^{-1} \right)^{-\frac{2m + n - 1}{m+n-1} }
\] and

\[
c_2 = \frac{(2m+n-2) \Lambda}{m+n-2} \left(\int_{M}w^{\frac{2(m+n-1)}{m+n-2}}v^{-1} \right)^{-\frac{m}{m+n-1} }.
\]
\end{pro}

\begin{proof} This proposition follows immediately from the fact that the conformal Laplacian is self-adjoint, and the definition of the weighted Escobar constant. \end{proof}

\begin{rem}
If $\Lambda = 0$ then we have in the proposition above that $c_1 = 0$, $c_2 = 0$. In this case, it follows that the equations in \eqref{Euler Lagrange Escobar quotient} coincide with the equations for finding a new conformal smooth metric measure space such that $\hat{R}^{m}_{\phi} \equiv 0$ and $\hat{H}^{m}_{\phi} \equiv 0$. Moreover, the problem to find a conformal smooth metric measure space with $\hat{R}^{m}_{\phi} \equiv 0$ and $\hat{H}^{m}_{\phi} \equiv C$ is solved by a direct compact argument on the functional

\[
\check{Q}(w) = \dfrac{(L^{m}_{\phi} w, w)_M + (B^{m}_{\phi} w, w)_{\partial M} }{ (\int_{\partial M } |w|^{\frac{2(m+n-1)}{m +n -2}} )^{\frac{m + n -2}{m+n-1}}}
\] due to $\frac{2(m+n-1)}{m +n -2}< \frac{2(n-1)}{n -2}$ for $m>0$.
\end{rem}

Next, we consider the Euler Lagrange equation on the $\mathcal{W}$-functional and  we will use it in the proof of Theorem B.

\begin{pro}\label{funcion realiza nu}
Let $(M^n,  g, v^{m} dV_g, v^{m} d \sigma_g)$ be a compact smooth metric measure space with boundary, fix
$\tau > 0$, and suppose that $w \in H^{1}(M)$ is a non-negative critical point of the map $\xi \to \mathcal{W}(\xi, \tau)$ acting on the space of volume-normalized elements of $H^{1}(M)$.
Then $w$ is a weak solution of

\begin{equation}\label{ecuacion tau}
\begin{array}{rclr}
\tau^{\frac{m}{2(m+n-1)}} L^{m}_{\phi}w + \frac{m + n - 1}{m+n-2} \tau^{-\frac{1}{2}} w^{\frac{m+n}{m+n-2}} v^{-1}& = &0 &\quad \text{in} \quad M, \\
\tau^{\frac{m}{2(m+n-1)}} B^{m}_{\phi}w & = & c_{3} w^{\frac{m+n}{m+n-2}} &\quad \text{on} \quad \partial M, \\
\end{array}
\end{equation} where

\[
c_3 = (\nu(\tau) + 1) + \frac{\tau^{-\frac{1}{2}}}{m+n-2} \int_{M}w^{\frac{2(m+n-1)}{m+n-2}}v^{-1}.
\]


If additionally $(w, \tau)$ is a minimizer of the energy, then


\begin{equation}\label{ecuacion tau 2}
c_3 = \frac{(m + n -1) (2m + n- 2)}{(m+n-2) (2m + n -1)} (\nu + 1).
\end{equation}

\end{pro}

\begin{proof}  The equality \eqref{ecuacion tau} follows immediately from the definition of $\mathcal{W}$. If $(w, \tau)$ is a critical point of the map $(w, \tau) \to \mathcal{W}(w, \tau)$, then

\begin{equation}
\frac{m}{m+n-1} \tau^{\frac{m}{2(m+n-1)}}((L^{m}_{\phi} w, w) + (B^{m}_{\phi} w, w)) =  \tau^{-\frac{1}{2}} \int_{M} w^{\frac{2(m+n-1)}{m+n-2}}v^{-1}.
\end{equation}

Using  this identity we can express $\nu$ and $c_3$ in terms of $(L^{m}_{\phi} w, w) + (B^{m}_{\phi} w, w)$ and these expressions yields  \eqref{ecuacion tau 2}. \end{proof}

\subsection{ \bf Euclidean half-space as the model space weighted Escobar problem}

Theorem \ref{teotrace} gives a complete classification of the minimizers  for the weighted Escobar quotient in the model space $(\R^n_{+}, dt^2 + dx^2 ,  dV, d \sigma, m)$ for $m$ non-negative integer.  In this sub-section we take a new $(\tau, x_0)$-parametric family of functions as in \eqref{funcion trace} with $\tau>0$ and $x_0 \in \R^{n-1}$.

To define the $(\tau, x_0)$-parametric family of functions fix $n \geq 3$ and $m > 0$. Given any $x_0 \in \R^{n-1}$ and $\tau > 0$, define the function $w_{x_{0}, \tau} \in C^{\infty}(\R^{n}_{+})$ by

\begin{equation}\label{funcion modelo}
w_{x_{0}, \tau} (t, x) = \tau^{-\frac{(n-1)(m+n-2)}{4(m+n-1)}} \left[ \left(1+ \left(\frac{c(m,n)}{\tau} \right)^{\frac{1}{2}}t\right)^2 + \frac{c(m,n) |x-x_0|^2}{\tau}  \right]^{-\frac{m+n-2}{2}}
\end{equation} where $c(m,n) = \frac{m+n-1}{m(m+n-2)^2}$. By change of variables we get

\begin{equation}\label{integral borde funcion modelo}
 V = \int_{\partial \R^{n}_{+}}  w_{x_{0}, \tau}^{\frac{2(m + n -1)}{m+n-2}} 1^{m} d \sigma = \int_{\partial \R^{n}_{+}}  w_{0, 1}^{\frac{2(m + n -1)}{m+n-2}} 1^{m} d\sigma.
\end{equation} A straightforward computation shows that

\begin{equation}\label{ecuacion funcion modelo}
\begin{array}{rclr}
-\tau^{\frac{m}{2(m+n-1)}} \Delta w_{x_{0}, \tau} + \frac{m + n - 1}{m+n-2} \tau^{-\frac{1}{2}} w_{x_{0}, \tau}^{\frac{m+n}{m+n-2}} & = &0 &\quad \text{in} \quad \R^{n}_{+}, \\
\tau^{\frac{m}{2(m+n-1)}} \dfrac{\partial w_{x_{0}, \tau}}{\partial \eta} & = & \left( \frac{m+n-1}{m}\right)^{\frac{1}{2}}w_{x_{0}, \tau}^{\frac{m+n}{m+n-2}} &\quad \text{on} \quad \partial \R^{n}_{+}, \\
\end{array}
\end{equation}

\begin{equation}\label{supremo funcion modelo}
\sup_{(x, t) \in \R^{n}_{+}} w_{x_0, \tau} (x, t) = w_{x_0, \tau}(x_0, 0) = \tau^{-\frac{(n-1)(m+n-2)}{4(m+n-1)}},
\end{equation} and for any $x \neq x_0$,

\begin{equation}\label{concentracion}
\lim_{\tau \to 0^{+}} w_{x_0, \tau} (x, t) = 0.
\end{equation}

Define $\tilde{w}_{x_0, \tau} = V ^{-\frac{m+n-2}{2(m+n-1)}}w_{x_0, \tau}$; with $V$ as in \eqref{integral borde funcion modelo}. Since $\tilde{w}_{x_0, \tau}$ achieves the weighted Escobar quotient, by Proposition \ref{W y Yamabe}, there exits $\tilde{\tau} > 0$ such that 

\small

\begin{equation}\label{nu Rn+ 2}
\begin{array}{l}
\nu(\R^{n}_+, dt^{2} + dx^{2},  dV, d \sigma, m)  + 1  = \mathcal{W}(\R^{n}_+, dt^{2} + dx^{2}, dV, d \sigma, m)(\tilde{w}_{x_0, \tau},  \tilde{\tau}) + 1 \\
 = \dfrac{\tilde{\tau}^{\frac{m}{2(m+n-1)}}} {V^{\frac{m+n-2}{m+n-1}} } \int_{\R^{n}_+} |\nabla w_{x_0, \tau}|^{2} dV + \tilde{\tau}^{-\frac{1}{2}} V^{-1} \int_{\R^{n}_+}  w_{x_0, \tau}^{\frac{2(m+n-1)}{m+n-2}}dV.
\end{array}
\end{equation}  

\normalsize

Then Proposition \ref{funcion realiza nu} yields $\tilde{\tau}= \tau V^{-\frac{2}{2m+n-1}}$.


\section{  The Escobar type problem for negative weighted Escobar constant}\label{Theorem A}

In this section, we prove Theorem A by a direct compactness argument. For this purpose, we develop some estimative for below for the Laplacian term in the Escobar quotient and some properties of Dirichlet eigenvalues and eigenfunctions.  In this section, $C$ is a real constant that depends only on the smooth metric measure space $(M^n,  g, v^{m} dV_g, v^{m} d \sigma_g)$ and possibly changing from line to line.

\subsection{\bf A below bound for conformal Laplacian term}

All functions in the family $\{w_{\epsilon, 0} \}$ as in \eqref{funcion trace} are minimizers of the weighted Escobar problem. Note that these functions are not uniformly bounded in $H^{1}(M)$ as $\epsilon \to 0$. That shows that in general there is no reason to find a minimizing function by direct arguments in the weighted Escobar quotient. It is possible that if the weighted Escobar quotient is finite and we try to minimize it with a sequence of functions normalized, then the terms involved in the numerator of the weighted Escobar quotient evaluated in these functions are not bounded uniformly. The next lemma deals with the  control of one of those terms from below.


\begin{lema}\label{estimativo por bajo}
Let $(M^n,  g, v^{m} dV_g, v^{m} d \sigma_g)$ be a compact smooth metric measure space with boundary and suppose that $\Lambda$ is finite, then there exists a real constant $C$ such that any volume-normalized function $\varphi \in H^{1}(M)$  satisfies

\begin{equation}\label{varphi}
(L^{m}_{\phi} \varphi, \varphi) + (B^{m}_{\phi} \varphi, \varphi) > C.
\end{equation}

\end{lema}

\begin{proof}  Suppose that  there exists a sequence of functions $\{ \varphi_i\}_{i = 1}^{\infty}$ such that

\begin{equation}\label{varphi_i}
\lim_{i \to \infty}(L^{m}_{\phi} \varphi_i, \varphi_i) + (B^{m}_{\phi} \varphi_i, \varphi_i) = - \infty \quad \text{and} \quad  \int_{\partial M}  \varphi_{i}^{\frac{2(m + n -1)}{m+n-2}} = 1.
\end{equation}

Since $\Lambda$ is finite there exists a real constant $C$ such that every volume-normalized $\varphi$  satisfies

\[
C \leq \Lambda(\varphi) = \left((L^{m}_{\phi} \varphi, \varphi) + (B^{m}_{\phi} \varphi, \varphi) \right) \left(\int_{ M}  \varphi^{\frac{2(m + n -1)}{m+n-2}} \right)^{\frac{m}{m+n-1}}.
\]

From the last inequality it follows that $\lim \limits_{i \to \infty} \int_{ M}  \varphi_{i}^{\frac{2(m + n -1)}{m+n-2}} = 0$ and by the H\"{o}lder inequality it follows that $\int_{M} \varphi_i^{2} < C$ for any $i$. Similarly, using that $\varphi_i$  are volume normalized and the H\"{o}lder inequality we  get $\int_{\partial M} \varphi_i^{2} < C$. Using these  $L^{2}$ estimate we obtain that

\[
(L^{m}_{\phi} \varphi_i, \varphi_i) + (B^{m}_{\phi} \varphi_i, \varphi_i) > C
\] contradicting the assumption \eqref{varphi_i}. \end{proof}

\subsection{\bf Dirichlet eigenvalues for the Conformal Laplacian}

In order to state the following lemma,  we say that a real number $\rho$ is an \textit{eigenvalue type Dirichlet} on $H^{1}_0(M) = \{ \varphi |\, \varphi \in H^{1}(M), \, \varphi \equiv 0 \, \, \text{on} \, \, \partial M\}$ if $\rho$ satisfies for some 
$\varphi \in H^{1}_0(M)$

\begin{equation}\label{ecuacion rho}
L^{m}_{\phi} \varphi = \rho \varphi \quad \text{in} \quad M,  \\
\quad \quad \varphi \equiv 0 \quad \text{on} \quad \partial M. 
\end{equation}

We also call $\varphi$ an \textit{eigenfunction} if it satisfies \eqref{ecuacion rho}. Let us denote by $\rho_1$  the first eigenvalue type Dirichlet on $H^{1,2}_0(M)$, then $\rho_1$ admits a variational characterization as

\begin{equation}\label{caracterizacion variacional rho}
\rho_1 = \inf \limits_{\varphi \in H^{1}_{0}(M)} \dfrac{\int_M |\nabla \varphi|^2 + \frac{m+n-2}{4(m+n-1)} R^{m}_{\phi} \varphi^2}{\int_{M} \varphi^{2}}.
\end{equation}

We have  $\rho_{1}$ is finite and we can choose an eigenfunction $\varphi$ associated to this eigenvalue such that $\varphi \geq 0$. Moreover, using the maximum principle we can take $\varphi > 0$ in $M \setminus \partial M$.

\begin{lema}\label{rho}
Let $(M^n,  g, v^{m} dV_g, v^{m} d \sigma_g)$ be a compact smooth metric measure space with boundary and $m > 0$. Then $\Lambda= -\infty$ if and only if $\rho_1 \leq 0$. 
\end{lema}

\begin{proof} First, let us assume $\rho_1 \leq 0$. Let $\varphi$ be a first eigenfunction of the problem \eqref{ecuacion rho} such that $\varphi > 0$ in $M \setminus \partial M$. Let us define 

\[
\psi_t = \dfrac{t \varphi + 1}{\sqrt{D}} \quad \text{where} \quad D = \left(\int_{\partial M} e^{-\phi} d\sigma_g \right)^{\frac{m+n-2}{(m+n-1)}}
\] and observe that for some constant $C > 0$ we have

\begin{equation}\label{integral interior y frontera}
\int_{\partial M} \psi_{t}^{\frac{2(m+n-1)}{m+n-2}} = 1 \quad \quad  \text{and}  \quad \quad \int_{ M} \psi_{t}^{\frac{2(m+n-1)}{m+n-2}} \geq C > 0.
\end{equation}


\begin{claim} \begin{equation}\label{laplaciano yendo infinito}
(L^{m}_{\phi} \psi_t, \psi_t) + (B^{m}_{\phi} \psi_t, \psi_t) \to -\infty \quad \text{when} \quad t \to \infty.
\end{equation}

\end{claim}

To prove this claim, we argue as Garcia and Mu\~{n}oz in \cite[Proposition 1]{GJ}. First, we consider the case $\rho_1 < 0$, using that $\psi_t \equiv 0$ on $\partial M$ we get

\small

\begin{equation}\label{laplaciano psi}
(L^{m}_{\phi} \psi_t, \psi_t) + (B^{m}_{\phi} \psi_t, \psi_t) = \frac{1}{D} \left[t^{2} \left(\rho_1 \int_{M} \varphi^2 \right) + t \left(\frac{m+n-2}{2(m+n-1)}  \int_{M}   \varphi R^{m}_{\phi} \right) + E \right]
\end{equation} \normalsize where

$$ E = \frac{m+n-2}{4(m+n-1)}  \int_{M} R^{m}_{\phi} + \frac{m+n-2}{2(m+n-1)}  \int_{M}   H^{m}_{\phi}.$$ Since  $\rho_1 < 0$, the quadratic term for $t$ on the right hand side of \eqref{laplaciano psi} is negative. Letting $t \to \infty$ it follows our claim in this case.

Now, we suppose that $\rho_1 = 0$, then

\begin{equation}\label{laplaciano psi2}
(L^{m}_{\phi} \psi_t, \psi_t) + (B^{m}_{\phi} \psi_t, \psi_t) = \frac{1}{D} \left[ t \left(\frac{m+n-2}{2(m+n-1)}  \int_{M}   \varphi R^{m}_{\phi} \right) + E \right]
\end{equation} \normalsize where $E$ is defined as in the previous case. Since $\varphi \equiv 0$ on $\partial M$, by Hopf's Lemma,  $\dfrac{\partial \varphi}{\partial \eta} < 0$. Then, integrating by parts yields

\[
\frac{m+n-2}{4(m+n-1)}  \int_{M}   \varphi R^{m}_{\phi} = \int_{M} \Delta_{\phi} \varphi =  \int_{\partial M} \dfrac{\partial \varphi}{\partial \eta} < 0.
\]

Then, the linear term for $t$ on the right hand side of \eqref{laplaciano psi2} is negative. Taking $t \to \infty$ we get the conclusion in this case and we finish the claim's proof.

Finally, from the estimates \eqref{integral interior y frontera} and \eqref{laplaciano yendo infinito} we get that $\mathcal{Q}(\psi_t) \to -\infty$ as $t \to \infty$, therefore we conclude $\Lambda = -\infty$.
 
Next, we assume that $\Lambda = -\infty$ and we prove that $\rho_1 \leq 0$. This assumption implies that $R^{m}_{\phi} $ is not identically zero. Let us take a minimizing sequence of functions $\{ \varphi_i\}_{i = 1}^{\infty}$ of $\Lambda$ such that 

\[
\int_{\partial M} \varphi_{i}^{\frac{2(m+n-1)}{m+n-2}} = 1, \quad (L^{m}_{\phi} \varphi_{i}, \varphi_{i}) + (B^{m}_{\phi} \varphi_{i}, \varphi_{i})  \leq 0 \quad \text{and}  \quad \lim \limits_{i \to \infty} \mathcal{Q}(\varphi_i) = - \infty.
\] 

\begin{claim} $\displaystyle \int_{ M} \varphi_{i}^{\frac{2(m+n-1)}{m+n-2}} \to  \infty$ when $i \to \infty$. \end{claim}

Arguing by contradiction, we assume that there exists a constant $C>0$ such that $\int_{ M} \varphi_{i}^{\frac{2(m+n-1)}{m+n-2}} < C$, then by the H\"{o}lder inequality we get that $\int_{ M} \varphi_{i}^{2} < C$ for every $i$. On the other hand,  we have that $(L^{m}_{\phi} \varphi_{i}, \varphi_{i}) + (B^{m}_{\phi} \varphi_{i}, \varphi_{i}) \to -\infty$ when $i \to \infty$ since $\lim \limits_{i \to \infty} \mathcal{Q}(\varphi_i) = - \infty$. Using this limit, the fact that $R^{m}_{\phi} $ is a non-zero function and that $\varphi_{i}$ is normalized we get $\int_{ M} \varphi_{i}^{2} \to  \infty$ when $i \to \infty$, which is a contradiction with the initial assumption. Hence  $\int_{ M} \varphi_{i}^{\frac{2(m+n-1)}{m+n-2}} \to  \infty$.

\begin{claim} $\displaystyle \int_{ M} \varphi_{i}^{2} \to  \infty$ when $i \to \infty$. \end{claim}

Arguing by contradiction, suppose that there exists a constant $C>0$ such that $\int_{ M} \varphi_{i}^{2} < C$. Then

\begin{equation}\label{desigualdad norma rho}
\int_M |\nabla \varphi_i|^2 \leq (L^{m}_{\phi} \varphi_i, \varphi_i) + (B^{m}_{\phi} \varphi_i, \varphi_i) + C(||\varphi_i||^2_{2, M} + ||\varphi_i||^2_{2, \partial M}) < C.
\end{equation}

On the other hand, by the Sobolev inequality we get that there exists a constant $C$ such that 

\begin{equation}\label{desigualdad norma rho M}
\int_{ M} \varphi_{i}^{\frac{2(m+n-1)}{m+n-2}} \leq C \left(\int_M |\nabla \varphi_i|^2 + \int_{ M} \varphi_{i}^{2} \right).
\end{equation}

Then inequalities \eqref{desigualdad norma rho} and \eqref{desigualdad norma rho M} yield $\int_{ M} \varphi_{i}^{\frac{2(m+n-1)}{m+n-2}} \leq C $. This is a contradiction with the Claim $2$ and we conclude that $\int_{ M} \varphi_{i}^{2} \to  \infty$ when $i \to \infty$.

Now we are able to conclude the proof of the lemma. For this purpose let us define the functions $\psi_i = \dfrac{\varphi_i}{||\varphi_i||_{2, M}}$. Arguing as in the last part of Proposition $1$ in Garcia and Mu\~{n}oz \cite{GJ}, we get that a sub-sequence $\psi_i$  converges weakly to a function $\psi$ in $H^{1}_{0}(M)$ such that $||\psi||_{2, M} = 1$ and

\[
\rho_1 \leq \int_M |\nabla \psi|^2 + \frac{m+n-2}{4(m+n-1)} R^{m}_{\phi} \psi^2 \leq \liminf \limits_{i \to \infty} (L^{m}_{\phi} \psi_i, \psi_i) + (B^{m}_{\phi} \psi_i, \psi_i) \leq 0. 
\]

\end{proof}

\subsection{ \bf Proof of Theorem A}
In this subsection we prove Theorem A using the before Lemmas presented in this section.

\textbf{\textit{Proof of Theorem A}}. Let  $\{ w_i\}_{i=1}^\infty$ be a sequence of positive functions such that $\int_{\partial M}  w_{i}^{\frac{2(m + n -1)}{m+n-2}} = 1$, $\mathcal{Q}(w_i) \leq 0$ and $\mathcal{Q}(w_i) \to  \Lambda$ when $i \to \infty$. Then

\begin{equation}\label{desigualdad norma 2}
0 \geq (L^{m}_{\phi} w_i, w_i) + (B^{m}_{\phi} w_i, w_i) \geq ||\nabla w_i||^2_{2, M} - C(||w_i||^2_{2, M} + ||w_i||^2_{2, \partial M}).
\end{equation}

First, we consider the case $||w_i^{2}||_{2, M} < C$, then the last inequality yields that $\{w_i\}_{i=1}^{\infty}$ are uniformly bounded in $H^{1}(M)$. Recall that $m> 0$, then $1< \frac{2(m+n-1)}{m+n-2} < \frac{2(n-1)}{n-2}$, i.e. $\frac{2(m+n-1)}{m+n-2}$ is less than the critical Trace's inequality exponent. 
By Sobolev's and Trace's embedding Theorems, there exists a function $w$ and a sub-sequence $\{w_i\}_{i = 1}^{\infty}$ which converges to $w$ in $L^{2}(M)$, $L^{\frac{2(m+n-1)}{m+n-2}}(M)$ and $L^{\frac{2(m+n-1)}{m+n-2}} (\partial M)$ and also $\{w_i\}_{i = 1}^{\infty}$ converges weakly to $w$ in $H^{1}(M)$. It follows that there exist a constant $C$ such that

\[
\int_{M}w^{\frac{2(m+n-1)}{m+n-2}}v^{-1} \geq C  \quad \text{and} \quad ||w||_{\frac{2(m+n-1)}{m+n-2}, \partial M} = 1. 
\]

Then by construction, $w$ minimizes the weighted Escobar quotient and by  Proposition \ref{ecuacion minimizador Yamabe}, $w$ is a non-negative weak solution of

\begin{equation}
\begin{array}{ll}
L^{m}_{\phi}w + c_{1}w^{\frac{m+n}{m+n-2}} v^{-1} = 0 &\quad \text{in} \quad M, \\
B^{m}_{\phi}w = c_{2}w^{\frac{m+n}{m+n-2}} &\quad \text{on} \quad \partial M. \\
\end{array}
\end{equation}

Since $1 < \frac{m+n-1}{m+n-2} < \frac{n-1}{n-2}$, the usual elliptic regularity argument for sub-critical equations allows us to conclude that $w$ is in fact smooth and positive, as we desired.

Following, we prove that we do not have the case when $||w_i||_{2, M} \to \infty$ is unbounded. Arguing by contradiction, we assume that $||w_i||_{2, M} \to \infty$ when $i \to \infty$. Consider the $L^2$ re-normalized sequence of functions $\tilde{w}_i = \frac{w_i}{||w_i||_{2, M}}$. It follows that $||\tilde{w}_i||_{\frac{2(m+n-1)}{m+n-2}, \partial M} \to 0 $ when $i \to \infty$. Since $\tilde{w}_i $ satisfy the inequality \eqref{desigualdad norma 2} for every $i$ we know that $\{\tilde{w}_i\}_{i = 1}^{\infty}$ is uniformly bounded in $H^{1,2}(M)$. 

By  Sobolev's and  Trace's embedding Theorems, there exists a function $w$ and a sub-sequence $\{\tilde{w}_i\}_{i = 1}^{\infty}$ which converges to $w$ in $L^{2}(M)$, $L^{\frac{2(m+n-1)}{m+n-2}}(M)$ and $L^{\frac{2(m+n-1)}{m+n-2}} (\partial M)$ and also weakly in $H^{1}(M)$. In consequence, $||w||_{2, M} = 1$ and using again that $||\tilde{w}_i||_{\frac{2(m+n-1)}{m+n-2}, \partial M} \to 0 $ when $i \to \infty$, we get that $w \equiv 0$ in $\partial M$. 

On the other hand, Lemma \ref{estimativo por bajo} yields  

\[
0 > (L^{m}_{\phi} w_i, w_i) + (B^{m}_{\phi} w_i, w_i) > -C.
\] 

Therefore $(L^{m}_{\phi} \tilde{w}_i, \tilde{w}_i) + (B^{m}_{\phi} \tilde{w}_i, \tilde{w}_i) \to 0$ when $i \to \infty$. Using $w$ as a test function in \eqref{caracterizacion variacional rho}, we conclude that 

\[
\rho_1 \leq \int_M |\nabla w|^2 + \frac{m+n-2}{4(m+n-1)} R^{m}_{\phi} w \leq \liminf \limits_{i \to \infty} (L^{m}_{\phi} \tilde{w}_i, \tilde{w}_i) + (B^{m}_{\phi} \tilde{w}_i, \tilde{w}_i) = 0.
\]

But $\rho_1 \leq 0$ contradicts Lemma \ref{rho} because  $\Lambda $ is finite by hypothesis.  $\bs$


\section{Aubin type inequality for weighted Escobar constants}\label{Theorem B}

In this section, we find an upper bound for the $\tau$-energy as $\tau$ goes to zero, Theorem B is a consequence of this estimate. To prove this estimate, we use Theorem \ref{teotrace} and the family $\{w_{0, \tau}\} $ in \eqref{funcion modelo} as test functions in the $\mathcal{W}$-functional. Actually, Theorem \ref{teotrace} is the reason for which the weighted Escobar constant for the Euclidean half-space appears on the right hand side of the inequality \eqref{desigualdad Jhova Aubin}. Similar ideas to prove Theorem B appeared in \cite{JhovannyAubin}. As in the previous section, $C$ is a real constant that depends only on the smooth metric measure space $(M^n,  g, v^{m} dV_g, v^{m} d \sigma_g)$ and possibly changing from line to line or in the same line.

\begin{lema}\label{limsup nu}
Let $(M^n,  g, v^{m} dV_g, v^{m} d \sigma_g)$ be a compact smooth metric measure space with  boundary and $m \geq 0$, then

\[
\limsup \limits_{\tau \to 0} \nu(\tau) \leq \nu[\R^{n}_{+},  dt^2 + dx^2, dV,  d \sigma, m].
\]
\end{lema}

\begin{proof} First define $\tilde{w}_{x_0, \tau} = V ^{-\frac{m+n-2}{2(m+n-1)}}w_{x_0, \tau}$; with $V$ as in \eqref{integral borde funcion modelo}. By Theorem \ref{teotrace} we know that  $\tilde{w}_{x_0, \tau}$ achieves the weighted Escobar quotient, hence by Proposition \ref{W y Yamabe}, there exits $\tilde{\tau} > 0$ such that 

\small

\begin{equation}\label{nu Rn}
\begin{array}{l}
\nu(\R^{n}_+, dt^{2} + dx^{2}, dV, d \sigma, m)  + 1  = \mathcal{W}(\R^{n}_+, dt^{2} + dx^{2},  dV, d \sigma, m)(\tilde{w}_{x_0, \tau},  \tilde{\tau}) + 1 \\
 = \dfrac{\tilde{\tau}^{\frac{m}{2(m+n-1)}}} {V^{\frac{m+n-2}{m+n-1}} } \int_{\R^{n}_+} |\nabla w_{x_0, \tau}|^{2} dV + \tilde{\tau}^{-\frac{1}{2}} V^{-1} \int_{\R^{n}_+}  w_{x_0, \tau}^{\frac{2(m+n-1)}{m+n-2}}dV.
\end{array}
\end{equation}  

\normalsize

Then Proposition \ref{funcion realiza nu} yields  $\tilde{\tau}= \tau V^{-\frac{2}{2m+n-1}}$.

On the other hand, fix a point $p \in \partial M$ and let $(x_i, t)$ be the Fermi coordinates in some fixed neighborhood $U$ of $p = (0, ..., 0)$. Let $1 > \epsilon > 0$ be such that $B(p, 2 \epsilon) \subset U$. Let $\eta : M \to [0, 1]$ be a cutoff function such
that $\eta \equiv 1$ on $B^{+}_{\epsilon}$, $supp (\eta) \subset B^+_{2\epsilon}$ and $|\nabla \eta|^{2} < C\epsilon^{-1}$ in $A^+_{\epsilon} = B^+_{2\epsilon} \smallsetminus B^+_{\epsilon}$. For each $0 < \tau < 1$, define $f_{\tau} : M \to \R$ by $f_{\tau} (x_1, . . . , x_{n-1}, t) = \eta w_{0,\tau} (x_1, . . . , x_{n-1}, t)$, and set  $\tilde{f}_{\tau} = V_{\tau}^{-\frac{m+n-2}{2(m+n-1)}} f_{\tau}$ for

\[
V_{\tau} = \int_{\partial M} f_{\tau}^{\frac{2(m+n-1)}{m+n-2}}.
\]

Proposition \ref{weighted Yamabe energy conformal} implies that if $w$ is a normalized function with the metric $v^{-2}g$, then

\small

\[
\mathcal{W}[M^n,  v^{-2}g,  dV_{v^{-2}g}, d \sigma_{v^{-2}g}, m](w, \tau) = \mathcal{W}[M^n, g, v^{m} dV_g, v^{m} d \sigma_g](v^{-\frac{m+n-2}{2}} w, \tau),
\] \normalsize this equality  allows us to consider without loss generality that $v \equiv 1$. Computing as in \cite[Lemma $3.4$]{LeeParker}, and using that $dV_{g} = (1 + O(r)) dxdt$  and $d\sigma_{g} = (1 + O(r)) dx$ we obtain

\small

\begin{equation}\label{w}
\begin{array}{l}
\mathcal{W}[M^n,  g, dV_{g}, d \sigma_{g}, m](\tilde{f}_{\tau}, \tilde{\tau}) + 1 \vspace{0.2cm}\\
= \displaystyle \dfrac{\tilde{\tau}^{\frac{m}{2(m+n-1)}}} {V_{\tau}^{\frac{m+n-2}{m+n-1}} } \left (\int_{B^{+}_{2\epsilon}} |\nabla f_{\tau}|_g^{2} + \frac{m+n-1}{4(m+n-2)}R_{g} f_{\tau}^2 dV_g \right. \vspace{0.2cm}\\
\displaystyle \left. + \int_{B^{+}_{2\epsilon} \cap \partial M} \frac{m+n-1}{2(m+n-2)}H_{g} f_{\tau}^2 d\sigma_g  \right) 
 + \tilde{\tau}^{-\frac{1}{2}} V_{\tau}^{-1} \int_{B^{+}_{2\epsilon}} f_{\tau}^{\frac{2(m+n-1)}{m+n-2}} dV_g \vspace{0.2cm}\\
\displaystyle \leq (1 + C\epsilon) \left\{ \dfrac{\tilde{\tau}^{\frac{m}{2(m+n-1)}}} {V_{\tau}^{\frac{m+n-2}{m+n-1}} } \left (\int_{B^{+}_{2\epsilon}} |\nabla f_{\tau}|_g^{2} + \frac{m+n-1}{4(m+n-2)}R_{g} f_{\tau}^2 dxdt \right. \right. \vspace{0.2cm}\\
\displaystyle \left. + \int_{B^{+}_{2\epsilon} \cap \partial M} \frac{m+n-1}{2(m+n-2)}H_{g} f_{\tau}^2 dx  \right) \left. + \tilde{\tau}^{-\frac{1}{2}} V_{\tau}^{-1} \int_{B^{+}_{2\epsilon}} f_{\tau}^{\frac{2(m+n-1)}{m+n-2}} dxdt \right\}.
\end{array}
\end{equation}

\normalsize

Let us recall that $c(m,n) = \frac{m+n-1}{m(m+n-2)^2}$. Fixing $\epsilon < 1$ and after taking $\sqrt{\tau} \leq \sqrt{c(m,n)}2\epsilon$ we obtain

\begin{equation}
\begin{array}{ll}
\displaystyle \int_{B^{+}_{2\epsilon}} R_{g} f_{\tau}^2 dxdt & \displaystyle \leq C \int_{B^{+}_{2\epsilon}} w_{0,\tau}^2 dxdt \vspace{0.2cm}\\
& = C\tau^{-\frac{(n-1)(m+n-2)}{2(m+n-1)}}   \displaystyle \int_{B^{+}_{2\epsilon}} \dfrac{ dxdt}{((1 + (\frac{c(m,n)}{\tau})^{\frac{1}{2}} t)^2 + \frac{c(m,n)}{\tau} |x|^2)^{m+n-2}} \vspace{0.2cm}\\
&= C\tau^{\frac{n-1}{2(m+n-1)}+\frac{1}{2}}  \displaystyle  \int_{B^{+}_{\frac{2\epsilon \sqrt{c(m,n)}}{\sqrt{\tau}}}} \dfrac{dydt}{((1 + s)^2 +  |y|^2)^{m+n-2}}. \\
\end{array}
\end{equation}

Similar as in \cite[Lemma $3.5$]{LeeParker} we get

\small

\begin{equation}
\displaystyle  \int_{B^{+}_{\frac{2\epsilon \sqrt{c(m,n)}}{\sqrt{\tau}}}} \dfrac{dydt}{((1 + s)^2 +  |y|^2)^{m+n-2}} =
\left \{ \begin{array}{ll}
C &  \text{if} \quad 4-n -2m < 0, \\
O(\tau^{m- \frac{1}{2}}) &  \text{if} \quad n = 3,  \, \frac{1}{2} - m > 0  \, \, \text{and}\\
O(\log(\tau)) & \text{if} \quad n = 3, \, \frac{1}{2} - m = 0. \\
\end{array}\right.
\end{equation}

\normalsize

Then 

\begin{equation}\label{r}
\displaystyle  \int_{B^{+}_{2\epsilon}} R_{g} f_{\tau}^2 dxdt  = E_1 =
\left \{ \begin{array}{ll}
O(\tau^{\frac{n-1}{2(m+n-1)}+\frac{1}{2}}) &  \text{if} \, 4-n -2m < 0, \\
O(\tau^{\frac{n-1}{2(m+n-1)}+ m} ) & \text{if} \, n = 3, \,  m < \frac{1}{2} \,  \text{and}  \\
O(\tau^{\tau^{\frac{n-1}{2(m+n-1)}+\frac{1}{2}}} \log(\tau)) & \text{if} \, n = 3,  \, \frac{1}{2} - m = 0. \\
\end{array}\right.
\end{equation}

Now, we estimate the  integrals on the right hand side in the second inequality of \eqref{w}

\begin{equation}\label{h}
\begin{array}{ll}
\displaystyle \int_{B^{n-1}_{2\epsilon} } H_{g} f_{\tau}^2 dx  &\leq C \int_{B^{n-1}_{2\epsilon}} w_{0,\tau}^2 dx = C\tau^{\frac{n-1}{2(m+n-1)}}\int_{B^{n-1}_{2\epsilon} } (1 + |y|^2)^{-(m+n-2)} dx \\
&\leq C\tau^{\frac{n-1}{2(m+n-1)}} 
\end{array}
\end{equation}

\begin{equation}\label{critico}
\int_{B^{+}_{2\epsilon}} f_{\tau}^{\frac{2(m+n-1)}{m+n-2}} dxdt  \leq \int_{\R^n_{+}} w_{0, \tau}^{\frac{2(m+n-1)}{m+n-2}} dxdt.
\end{equation}

Let us  estimate the  gradient integral in $A^+_{\epsilon} = B^+_{2\epsilon} \smallsetminus B^+_{\epsilon}$. Observe that 

\begin{equation}\label{gradiente en el anillo}
|\nabla f_{\tau}|_{\tilde{g}}^{2} \leq C|\nabla f_{\tau}|^{2} \leq C (\eta^{2}|\nabla w_{0, \tau}|^{2} + |\nabla \eta|^{2} w_{0, \tau}^{2} ).
\end{equation}

Now, we get 

\begin{equation}
\begin{array}{ll}
\displaystyle \int_{A^+_{\epsilon}} |\nabla \eta|^{2} w_{0, \tau}^{2} dxdt & \displaystyle \leq  C \epsilon^{-2} \int_{A^+_{\epsilon}} w_{0, \tau}^{2} dxdt \vspace{0.2cm}\\
&\leq  C \epsilon^{-2} \tau^{\frac{-(n-1)(m+n-2)}{2(m+n-1)} +\frac{n}{2}} \displaystyle  \int_{A^+_{\frac{\epsilon \sqrt{c(m,n)}}{\sqrt{\tau}} } } \left( \dfrac{1}{s^2 + |y|^{2}}\right)^{m+n-2} dxdt \vspace{0.2cm}\\
& \leq C \epsilon^{2-n -2m} \tau^{\frac{n-1}{2(m+n-1)} + m +\frac{n -3}{2} }
\end{array}
\end{equation} and

\begin{equation}
\begin{array}{ll}
\displaystyle \int_{A^+_{\epsilon}} \eta^{2}|\nabla w_{0, \tau}|^{2}dxdt & \leq  \displaystyle C \tau^{\frac{-(n-1)(m+n-2)}{2(m+n-1)} +\frac{n}{2} -1} \int_{A^+_{\frac{\epsilon \sqrt{c(m,n)}}{\sqrt{\tau}} } } \left( \dfrac{1}{s^2 + |y|^{2}}\right)^{m+n-1} dxdt\vspace{0.2cm}\\
& \leq C \epsilon^{2-n -2m} \tau^{\frac{n-1}{2(m+n-1)} + m +\frac{n -3}{2} }.
\end{array}
\end{equation}

Then

\begin{equation}\label{gradiente 2epsilon}
\begin{array}{ll}
\displaystyle \int_{A^+_{\epsilon}}  |\nabla f_{\tau}|_{g}^{2}dxdt & \leq C \epsilon^{2-n -2m} \tau^{\frac{n-1}{2(m+n-1)} + m +\frac{n -3}{2} }.
\end{array}
\end{equation}

Since for the Fermi coordinates around $p$ we obtain $g^{tt} = 1$, $g^{ti} = 0$ and $g^{ij} = \delta_{ij} + O(|x, t|)$ where $1 \leq i, j \leq n-1$, it follows 

\begin{equation}\label{gradiente epsilon}
\begin{array}{ll}
\displaystyle \int_{B_{\epsilon}} |\nabla f_{\tau}|_g^{2} dxdt & \displaystyle \leq \int_{B_{\epsilon}} |\nabla w_{0, \tau}|^{2} dxdt + C\int_{B_{\epsilon}} |x,t| (w_{0, \tau})_i (w_{0, \tau})_j \vspace{0.2cm}\\
& \displaystyle  \leq \int_{B_{\epsilon}} |\nabla w_{0, \tau}|^{2} dxdt + C\tau^{\frac{n-1}{2(m+n-1)}}.
\end{array}
\end{equation}

We already have the second inequality of \eqref{gradiente epsilon} because 

\small

\begin{equation}
\begin{array}{ll}
\displaystyle \int_{B_{\epsilon}} |x,t| (w_{0, \tau})_i (w_{0, \tau})_j & \leq C \tau^{-\frac{(n-1)(m+n-2)}{2(m+n-1)} - 2} \displaystyle \int_{B_{\epsilon}}  \dfrac{|x,t|x_i x_j dxdt }{((1 + (\frac{c(m,n)}{\tau})^{\frac{1}{2}} t)^2 + \frac{c(m,n)}{\tau} |x|^2)^{m+n}} \vspace{0.2cm}\\
& \leq C\tau^{\frac{n-1}{2(m+n-1)}}  \displaystyle  \int_{B^{+}_{\frac{2\epsilon \sqrt{c(m,n)}}{\sqrt{\tau}}}} \dfrac{|y,s|^3 dydt}{((1 + s)^2 +  |y|^2)^{m+n}} \vspace{0.2cm}\\
& \leq C\tau^{\frac{n-1}{2(m+n-1)}}. 
\end{array}
\end{equation}

\normalsize

Using the inequalities \eqref{r}, \eqref{h}, \eqref{gradiente 2epsilon} and \eqref{gradiente epsilon} in the inequality \eqref{w} we get that

\small

\begin{equation}\label{w2+}
\begin{array}{l}
\mathcal{W}[M^n,  g,  dV_{g}, d \sigma_{g}, m](\tilde{f}_{\tau}, \tilde{\tau}) + 1 \\
\displaystyle \leq (1 + C\epsilon) \left\{ \dfrac{\tilde{\tau}^{\frac{m}{2(m+n-1)}}} {V_{\tau}^{\frac{m+n-2}{m+n-1}} } \left ( \int_{\R^n_{+}} |\nabla w_{0, \tau}|^{2} dxdt + C\tau^{\frac{n-1}{2(m+n-1)}}  \right.  \right.\vspace{0.2cm} \\
\displaystyle \left. + C \tau^{\frac{(n-1)(2m+n-1)}{2(m+n-1)} + m +\frac{n -3}{2} } \epsilon^{2-n -2m} + E_1   \right) \left. + \tilde{\tau}^{-\frac{1}{2}} V_{\tau}^{-1} \int_{\R^n_{+}} w_{0, \tau}^{\frac{2(m+n-1)}{m+n-2}} dxdt \right\}.
\end{array}
\end{equation}

\normalsize

Now using the inequality \eqref{nu Rn} we conclude

\small

\begin{equation}\label{w2}
\begin{array}{l}
\mathcal{W}[M^n,  g,  dV_{g},  d \sigma_{g}, m](\tilde{f}_{\tau}, \tilde{\tau}) + 1 \\
\leq (1 + C\epsilon) \nu[\R^{n}_{+},  dt^2 + dx^2, dV_g, d \sigma_g, m]   \vspace{0.2cm}\\
+ (1 + C\epsilon) \left\{   \tilde{\tau}^{\frac{m}{2(m+n-1)}} V_{\tau}^{-\frac{m+n-2}{m+n-1}} \left(  C\tau^{\frac{n-1}{2(m+n-1)}} + C \tau^{\frac{(n-1)(2m+n-1)}{2(m+n-1)} + m + \frac{n-3}{2}}  \epsilon^{2-n -2m}  + E_1   \right) \right. \vspace{0.2cm}\\
\displaystyle  + \tilde{\tau}^{\frac{m}{2(m+n-1)}} (V_{\tau}^{-\frac{m+n-2}{m+n-1}} - V^{-\frac{m+n-2}{m+n-1}}) \int_{\R^n_{+}} |\nabla w_{0, \tau}|^{2} dxdt \vspace{0.2cm} \\
\displaystyle \left. + \tilde{\tau}^{-\frac{1}{2}} (V_{\tau}^{-1} - V^{-1}) \int_{\R^n_{+}} w_{0, \tau}^{\frac{2(m+n-1)}{m+n-2}} dxdt\right\}.
\end{array}
\end{equation}

\normalsize

 On the other hand, we obtain

\begin{equation}\label{V - Vt}
\begin{array}{ll}
V - V_{\tau} & \leq  \displaystyle \int_{\R^{n-1} \setminus B_{\epsilon}^{n-1}} w_{0, \tau}^{\frac{2(m+n-1)}{m+n-2}} dx \vspace{0.2cm}\\
& = \displaystyle \tau ^{-\frac{n-1}{2}}\int_{\partial \R^n_{+} \setminus B^{n-1}_{\epsilon}} (1 + \frac{c(m,n)}{\tau}|x|^2)^{-(m+n-1)} dx \vspace{0.2cm}\\
& = \displaystyle  C \int_{\partial \R^n_{+} \setminus B^{n-1}_{\frac{2\epsilon \sqrt{c(m,n)}}{\sqrt{\tau}}}} (1 + |y|^2)^{-(m+n-1)} dy \vspace{0.2cm}\\
& \leq C \epsilon^{1-n-2m}\tau^{m + \frac{n}{2} - \frac{1}{2}}.
\end{array}
\end{equation}

In particular, we get that the constants $V_{\tau}$ are uniformly bounded away from zero. Using  estimate \eqref{V - Vt} and the Taylor expansion for the functions $x^{-\frac{m+n-2}{m+n-1}}$  and $x^{-1}$ we obtain

\begin{equation}\label{V - Vt potencia}
\begin{array}{ll}
V_{\tau}^{-\frac{m+n-2}{m+n-1}} - V^{-\frac{m+n-2}{m+n-1}} \leq C \epsilon^{1-n-2m}\tau^{m + \frac{n}{2} - \frac{1}{2}}
\end{array}
\end{equation} and 

\begin{equation}\label{V - Vt inversos}
\begin{array}{ll}
V_{\tau}^{-1} - V^{-1} \leq C \epsilon^{1-n-2m}\tau^{m + \frac{n}{2} - \frac{1}{2}}.
\end{array}
\end{equation}

Additionally, the equality \eqref{nu Rn} implies the following  estimates

\begin{equation}\label{terminos restantes}
\begin{array}{ll}
 \displaystyle \tilde{\tau}^{\frac{m}{2(m+n-1)}} \int_{\R^n_{+}} |\nabla w_{0, \tau}|^{2} dxdt \leq C \quad \text{and} \quad \tilde{\tau}^{-\frac{1}{2}} \int_{\R^n_{+}} w_{0, \tau}^{\frac{2(m+n-1)}{m+n-2}} dxdt  \leq C.
\end{array}
\end{equation}

The substitution $\tilde{\tau} = \tilde{\tau}V^{\frac{1}{2m+n-1}}$, the inequalities \eqref{V - Vt potencia}, \eqref{V - Vt inversos}, \eqref{terminos restantes} and \eqref{w2} yield 

\small

\begin{equation}\label{w3}
\begin{array}{l}
\mathcal{W}[M^n,  g,  dV_{g}, d \sigma_{g}, m](\tilde{f}_{\tau}, \tilde{\tau}) + 1 \vspace{0.2cm}\\
\leq (1 + C\epsilon) \nu[\R^{n}_{+},  dt^2 + dx^2, 1^{m} dV_g, 1^{m} d \sigma_g]   \vspace{0.2cm} \\
+ (1 + C\epsilon) \left\{ V^{-\frac{m+n-2}{m+n-1} -\frac{m}{2(2m+n-1)(m+n-1)}} \left(  C\tau^{\frac{1}{2}} + C \tau^{\frac{1}{2} + m +\frac{n -3}{2} } \epsilon^{2-n -2m}  \right. \right. \vspace{0.2cm}\\
\left. +  \tau^{\frac{m}{2(m+n-1)}}E_1   \right)\left. + C \epsilon^{1-n-2m}\tau^{m + \frac{n}{2} - \frac{1}{2}}  \right\}.
\end{array}
\end{equation}

\normalsize


Finally, taking $\tau \to 0$ and after $\epsilon \to 0$ in \eqref{w3}  the conclusion follows. \end{proof}

\textbf{\textit{Proof of Theorem B}}. By the definition of $\nu$ and Lemma \ref{limsup nu} we obtain that

\begin{equation}
\nu[M^n,  g, v^{m} dV_{g}, v^{m} d \sigma_{g}] \leq \nu[\R^{n}_{+},  dt^2 + dx^2,  dV, d \sigma, m].
\end{equation}

By Proposition \ref{W y Yamabe} we conclude 

\begin{equation}\label{Lambda}
\begin{array}{l}
\Lambda[M^n,  g, v^{m} dV_{g}, v^{m} d \sigma_{g}] \leq \Lambda[\R^{n}_{+},  dt^2 + dx^2,  dV, d \sigma, m]. \quad \bs\\
\end{array}
\end{equation}

\bigskip
\small\noindent{\bf Acknowledgments:}
This work was part of my  Ph.D. thesis under the guidance of Fernando Cod\'a Marques  at the Instituto Nacional de Matemática Pura e Aplicada (IMPA), Rio de Janeiro, Brazil. I am grateful to my advisor  for encouraging me to study this problem. I am also grateful to   Universidad del Valle, Cali, Colombia for the support and to  Capes for scholarship.

\end{document}